\documentclass{amsart}

\usepackage{enumerate}
\usepackage{amssymb}
\usepackage{graphicx}
\usepackage{MnSymbol}

\usepackage{amsthm}

\title[Independence of ST from BPIT]{The Independence of Stone's theorem from the Boolean Prime Ideal Theorem}
\author{Samuel M. Corson}

\theoremstyle{definition}\newtheorem{theorem}{Theorem}

\theoremstyle{definition}

\theoremstyle{definition}
\theoremstyle{definition}\newtheorem{proposition}[theorem]{Proposition}
\theoremstyle{definition}\newtheorem{definition}[theorem]{Definition}
\theoremstyle{definition}
\theoremstyle{definition}
\theoremstyle{definition}
\theoremstyle{definition}
\theoremstyle{definition}\newtheorem{lemma}[theorem]{Lemma}
\theoremstyle{definition}
\theoremstyle{definition}
\theoremstyle{definition}
\theoremstyle{definition}
\theoremstyle{definition}
\theoremstyle{definition}

\newcommand{\Aut}{\operatorname{Aut}}
\newcommand{\ST}{\textbf{ST}}
\newcommand{\BPIT}{\textbf{BPIT}}
\newcommand{\fix}{\operatorname{fix}}
\newcommand{\stab}{\operatorname{stab}}
\newcommand{\U}{\mathbb{U}}

\newcommand{\Cov}{\operatorname{Cov}}
\newcommand{\Ref}{\operatorname{Ref}}
\newcommand{\Inj}{\operatorname{Inj}}

\begin{document}

\address{Instituto de Ciencias Matem\'aticas CSIC-UAM-UC3M-UCM, 28049 Madrid, Spain.}

\email{sammyc973@gmail.com}
\keywords{paracompact, metacompact, metric space, Boolean Prime Ideal Theorem}
\subjclass[2010]{03E25, 54A35, 54E35, 54D20}
\thanks{This work was supported by ERC grant PCG-336983 and by the Severo Ochoa Programme for Centres of Excellence in R\&D SEV-20150554.}

\begin{abstract}  We give a permutation model in which Stone's Theorem (every metric space is paracompact) is false and the Boolean Prime Ideal Theorem (every ideal in a Boolean algebra extends to a prime ideal) is true.  The erring metric space in our model attains only rational distances and is not metacompact.  Transfer theorems give the comparable independence in the Zermelo-Fraenkel setting, answering a question of Good, Tree and Watson.


\end{abstract}

\maketitle

\begin{section}{Introduction}

Let $X$ be a topological space.  If $\mathcal{V}$ is an open cover of $X$ we say that open cover $\mathcal{O}$ is a \emph{refinement} of $\mathcal{V}$ if for each $O\in \mathcal{O}$ there exists $V\in \mathcal{V}$ for which $O\subseteq V$.  An open cover $\mathcal{O}$ is \emph{point finite} if for each $x\in X$ the set $\{O\in \mathcal{O}\mid x\in O\}$ is finite and is \emph{locally finite} if for each $x\in X$ there exists an open neighborhood $V$ of $x$ for which $\{O\in \mathcal{O}\mid O \cap V\neq \emptyset\}$ is finite.  We say $X$ is \emph{metacompact} if every open cover has a point finite refinement and is \emph{paracompact} if each open cover has a locally finite refinement.

The Boolean Prime Ideal Theorem, which we denote $\BPIT$, states that every prime ideal in a non-degenerate Boolean algebra extends to a prime ideal.  $\BPIT$ is equivalent to many theorems in mathematics such as the ultrafilter lemma, the existence of a Stone-\v{C}ech compactification for every Tychonov space, and the completeness theorem for first order logic (see \cite[Form 14]{HR}).  $\BPIT$ is also known to be strictly weaker than the axiom of choice \cite{HKRS} and for these reasons $\BPIT$ furnishes a natural benchmark against which to gauge the set theoretic strength of a theorem.

A classical theorem of A. H. Stone, which we denote $\ST$, states that metric spaces are paracompact \cite{St}.  Stone's proof, and subsequence proofs, of $\ST$ use the axiom of choice.  Good, Tree and Watson have shown that $\ST$  is independent of the Zermelo Fraenkel axioms plus the principle of dependent choices.  They ask whether $\ST$ is independent of $\BPIT$ \cite[p. 1216]{GTW}.  We answer this question in the affirmative, initially in the Fraenkel-Mostowski setting (ZFA denotes Zermelo-Fraenkel set theory with atoms).

\begin{theorem}\label{main}  There exists a permutation model $\mathcal{N}$ of ZFA in which

\begin{enumerate}

\item $\BPIT$ holds;

\item $\ST$ fails (there exists a rational-valued metric space which is not metacompact).

\end{enumerate}

\end{theorem}

We define $\mathcal{N}$ and check the listed properties in Section \ref{othersection}.  We then argue that the independence  transfers to the ZF setting.  For related results regarding the set theoretic strength of assertions in metric topology the interested reader can consult \cite{HKRS}, \cite{KT}, \cite{Ta}.
\end{section}

\begin{section}{The model $\mathcal{N}$ and its properties}\label{othersection}

The permutation model is constructed using a set of atoms corresponding to a countable universal object (a Fra\"iss\'e limit).  We say a metric space is \emph{$\mathbb{Q}$-valued} if all distances between points are elements in $\mathbb{Q}$.  A triple $(X, d, \prec)$ is a \emph{totally ordered $\mathbb{Q}$-valued metric space} if $(X, d)$ is a $\mathbb{Q}$-valued metric space and $(X, \prec)$ is a totally ordered set.  Let $\U_{\mathbb{Q}}^<$ denote the countable ultrahomogeneous totally ordered $\mathbb{Q}$-valued metric space.  In other words $\U_{\mathbb{Q}}^< = (U, d, <)$ is a totally ordered $\mathbb{Q}$-valued metric space, with $U$ countable, which satisfies the following:

\begin{enumerate}

\item  If $(X, d, \prec)$ is a finite totally ordered $\mathbb{Q}$-valued metric space then there exists an embedding $f: (X, d, \prec) \rightarrow \U_{\mathbb{Q}}^<$.

\item  If finite $X, Y \subseteq U$ are such that there exists an isomorphism $f:(X, d\upharpoonright (X\times X), <\upharpoonright (X\times X)) \rightarrow (Y, d\upharpoonright (X\times X), <\upharpoonright (Y \times Y))$ then there exists an automorphism $\phi: \U_{\mathbb{Q}}^< \rightarrow \U_{\mathbb{Q}}^<$ which extends $f$.
\end{enumerate}

The structure $\U_{\mathbb{Q}}^<$ can be constructed by induction using a back-and-forth argument, from the important fact that the class of $\mathbb{Q}$-valued totally ordered finite metric spaces is a  Fra\"iss\'e class (this is noted in \cite[page 111]{KPT}, for example).  The reader can see \cite[Section 2]{KPT} for a review of  Fra\"iss\'e theory.

We let $\mathcal{M}$ be a model of ZFA + AC with a countably infinite set $A$ of atoms.  Index the set of atoms $A = \{a_x\}_{x\in U}$ and endow $A$ with binary relation $\prec$ as well as binary relations $q$ for each $q\in \mathbb{Q} \cap [0, \infty)$  with $a_{x_0}\prec a_{x_1}$ if and only if $x_0 < x_1$ and $q(a_{x_0}, a_{x_1})$ if and only if $d(x_0, x_1) = q$.

Let $G$ be the group $\Aut(A)$ of all bijections on $A$ which preserve the binary relations above.  For each $B \subseteq A$ we let $\fix(B) = \{\phi \in G\mid (\forall a\in B)\phi(a) = a\}$.  Let $\mathcal{F}$ be the normal filter generated by $\fix(B)$ for each finite $B\subseteq A$.  Let $\mathcal{N}$ be the ZFA submodel of $\mathcal{M}$ of hereditarily $\mathcal{F}$-symmetric sets (those sets $X$ for which each element $Y$ of the transitive closure of $X$ has some finite $B_Y\subseteq A$ for which $\fix(B_Y) \leq \stab(Y):= \{\phi\in G\mid \phi(Y) = Y)\}$).  We check the two assertions of Theorem \ref{main} and then the transferability of the independence result.

\begin{subsection}{$\BPIT$ holds in $\mathcal{N}$}

We utilize a result of Blass \cite{B3} and a recent result due to Ne\v{s}et\v{r}il \cite{N}.  Recall that a Hausdorff topological group $G$ is \emph{extremely amenable} if for every continuous action $G \curvearrowright X$ on a nonempty compact Hausdorff space $X$ there is a universal fixed point (i.e. there is some $x\in X$ for which $g\cdot x = x$ for all $g\in G$).

By using a single direction of each of the equivalences in \cite[Theorem 5.1]{B3} and \cite[Theorem 5.2]{B3} we obtain the following:

\begin{proposition}\label{Blasscombined}  Let $\mathcal{P}$ be a permutation model of ZFA defined by set $A$ of atoms, nontrivial group $G$ of permutations on $A$ and normal filter $\mathcal{F}$ obtained by the ideal of finite sets.  Suppose also that by endowing $G$ with the topology arising from the normal filter $\mathcal{F}$ we have that $G$ is extremely amenable.  Then $\mathcal{P}$ satisfies $\BPIT$.
\end{proposition}

 We note the following which is due to combining the well known Ramsey result of Ne\v{s}et\v{r}il (see \cite[Section 5, Remark 2]{N}) with \cite[Theorem 4.7]{KPT}, see \cite[page 111]{KPT} or the introduction of \cite{Th}:

\begin{proposition}  The group $G \simeq \Aut(U_{\mathbb{Q}}^<)$ is extremely amenable.
\end{proposition}

\noindent That $\BPIT$ holds in $\mathcal{N}$ is now immediate.

\end{subsection}

\begin{subsection}{$\ST$ fails in $\mathcal{N}$}

We endow the set of atoms $A$ with the metric $\delta: A \times A \rightarrow \mathbb{Q}$ given by $\delta(a_0, a_1) = q$ if and only if $q(a_0, a_1)$.  Clearly this metric is hereditarily $\mathcal{F}$-symmetric and so the metric space $(A, \delta)$ is an element of $\mathcal{N}$.  Let $\mathcal{V} =  \{N(a, \frac{1}{2})\}_{a\in U}$ be the collection of open balls of radius $\frac{1}{2}$, which is also obviously in $\mathcal{N}$.

Suppose for contradiction that there exists a point finite refinement $\mathcal{O}$ of $\mathcal{V}$.  Let $B\subseteq A$ be such that $\fix(B) \leq \stab(\mathcal{O})$.  We can assume without loss of generality that $B$ is nonempty.  Let $D$ denote the diameter of the set $B$.  Select an element $a\in A \setminus B$ such that $\delta(a, a') =  D + 4$ and $a' \prec a$ for every $a'\in B$.

Let $n\in \omega \setminus \{0\}$ be given.  Since $\mathcal{O}$ is a cover of $(A, \delta)$ we select $O\in \mathcal{O}$ such that $a\in O$.  As $O$ is open we can select $m\in \omega$ for which $N(a, \frac{1}{m}) \subseteq O$.  Let $K = nm$.  Select $a_0, \ldots a_{3K} = a$ in $A$ be such that

\begin{itemize}

\item $a' \prec a_0 \prec a_1 \prec \cdots \prec a_{3K}$ for all $a'\in B$;

\item $\delta(a_i, a') = D +4$ for each $a'\in B$ and $0\leq i\leq 3K$; and

\item $\delta(a_i, a_j) = \frac{|i - j|}{K}$ for $0\leq i, j \leq 3K$.

\end{itemize}

\noindent Notice that $a_{3K - n + 1}, a_{3K - n + 2}, \ldots a_{3K} \in N(a_{3K}, \frac{1}{m}) \subseteq O$.  Let $\phi \in \fix(B)$ be such that $\phi(a_i) = a_{i+1}$ for every $0 \leq i < 3K$.  Now $a_{3K} = a \in \phi^j(O)$ for each $0 \leq j < n$.

If $J \subseteq U$ is any set of diameter at most $1$ with $a\in J$ we also know that $$\{a_0, \ldots, a_{3K}\}\cap J \subseteq \{a_{2K}, a_{2K + 1}, \ldots, a_{3K}\}.$$  Letting $0\leq L \leq 3K$ be minimal such that $a_L\in O$ we see that $2K \leq L \leq 3K - n +1$.  Moreover for every $0 \leq j < n$ we have

\begin{itemize}

\item $a_{L+ j}\in \phi^j(O)$; and

\item $a_i\notin \phi^j(O)$ for $K \leq i < L+ j$
\end{itemize}

\noindent where the latter assertion holds by noticing that if $a_i \in \phi^j(O)$ and $K \leq i <L + j$ then $\phi^{-j}(a_i) = a_{i-j} \in O$, but $0 \leq K - n < K - j \leq i - j < L$ and this contradicts the minimality of $L$.  Thus for $0\leq j < n$ we have $$\{a_K, \ldots a_{L+j}\} \cap \phi^j(O) = \{a_{L + j}\}$$ which demonstrates that the function $f = \{(j, \phi^j(O)\}_{0\leq j <n}$ is injective and  $a \in \phi^j(O) \in \mathcal{O}$ for each $0\leq j < n$.  That this function is in $\mathcal{N}$ is easily checked.

Thus we have demonstrated that in $\mathcal{N}$ for every $n\in \omega \setminus \{0\}$ there exists an injection $f: n \rightarrow \{O\in \mathcal{O}\mid a\in O\}$ and thus $A$ is a $\mathbb{Q}$-valued metric space which has an open cover which has no refinement which is point finite.
\end{subsection}

\begin{subsection}{Transfer to ZF}

We note that the negation of Stone's Theorem has already been recognized to be transferrable (see \cite[page 387, third line from bottom]{HR}).  We provide an argument for the sake of completeness.  For more background in transfer principles see \cite[Note 103]{HR}.

If $\mathcal{P}$ is a model of ZFA and $Y$ is a set in $\mathcal{P}$ we define

\begin{center}  $R_0(Y) = Y$

$R_{\alpha + 1}(Y) = P(R_{\alpha}(Y))\cup R_{\alpha}(Y)$

$R_{\alpha}(Y) = \bigcup_{\beta < \alpha} R_{\alpha}(Y)$ for $\alpha$ a non-zero limit ordinal.

\end{center}

\noindent where $P(X)$ denotes the powerset of $X$.

We specialize a definition from \cite[page 722]{P3} as is done in \cite[Definition 2.14]{Kl}.

\begin{definition}  Let $\overline{x} = (x_0, \ldots, x_{n-1})$ be a tuple of variables.  We'll say a formula $\Phi(\overline{x})$ is \emph{ordinal boundable} if for some absolutely definable ordinal $\alpha$ we have that

\begin{center} $\Phi(\overline{x}) \Longleftrightarrow \Phi^{R_{\alpha}(\bigcup \overline{x})}(\overline{x})$
\end{center}

\noindent is a theorem of ZFA.  A statement is \emph{ordinal boundable} if it is the existential closure of an ordinal boundable formula.
\end{definition}

Readers who are familiar with boundable statements \cite{HR} will quickly notice that ordinal boundable statements are boundable.

\begin{lemma}  The statement ``There exists a $\mathbb{Q}$-valued metric space having an open cover with no point finite refinement'' is ordinal boundable.
\end{lemma}

\begin{proof}  We give an ordinal bound using a succession of very naive bounds.  Certainly if $X \subseteq Y$ we have $R_{\alpha}(X) \subseteq R_{\alpha}(Y)$.  We have $\omega \in R_{\omega + 1}(\emptyset)$ and by constructing $(\omega, +)$, then $\mathbb{Z}$, then $(\mathbb{Z}, +)$, then $\mathbb{Q}$ and finally $(\mathbb{Q}, +)$ in the standard way we see that for example $(\mathbb{Q}, +)\in R_{\omega + 30}(\emptyset)$.

Given a set $X$ we have $X\times X \in R_{2}(X)$ and so $X\times X, (\mathbb{Q}, +)\in R_{\omega + 30}(X)$.  A function $d: X\times X \rightarrow \mathbb{Q}$ will satisfy $d\in R_{\omega + 33}(X)$.  A collection $\mathcal{V}$ of subsets of $X$ (for example an open cover) will satisfy $\mathcal{V} \in R_2(X)$.  Thus an ordered triple $(X, d, \mathcal{V})$ with $d$ a $\mathbb{Q}$-valued metric on $X$ and $\mathcal{V}$ being an open cover of $X$ under the topology induced by $d$ will satisfy, say, $(X, d, \mathcal{V}) \in R_{\omega + 37}(X)$.  A function $f$ from a natural number to an open cover of $X$ will satisfy $f\in R_{\omega + 41}(X)$.

Let $\Cov(X, d, \mathcal{U})$ denote that $d$ is a metric on $X$ and that $\mathcal{U}$ is an open cover of $X$ with respect to the metric $d$.  Let $\Ref(X, d, \mathcal{U}, \mathcal{V})$ denote that $\Cov(X, d, \mathcal{U})$ and $\Cov(X, d, \mathcal{V})$ and that $\mathcal{V}$ is a refinement of $\mathcal{U}$.  Let $\Inj(f, Y, Z)$ denote that $f$ is an injective function from $Y$ to $Z$.

Let $\Phi(X, d, \mathcal{U})$ signify

\begin{center}  $\Cov(X, d, \mathcal{U}) \wedge$

$(\forall \mathcal{V}\in P(P(X)))[\Ref(X, d, \mathcal{U}, \mathcal{V}) \rightarrow (\exists x\in X)(\forall n\in \omega)(\exists f \subseteq n\times \mathcal{V})[\Inj(f, n, \mathcal{V}) \wedge (\forall m< n)x\in f(m)]]$

\end{center}

\noindent The existential closure of $\Phi(X, d, \mathcal{U})$ is the statement in question and ZFA implies that

\begin{center}  $\Phi(X, d, \mathcal{U}) \Longleftrightarrow \Phi^{R_{\omega + 41}(X \cup d \cup \mathcal{U})}(X, d, \mathcal{U})$

\end{center}

\noindent which completes the proof.

\end{proof}

The independence of $\ST$ from $\BPIT$ in the Zermelo-Fraenkel setting now follows from the following consequence of a result of Pincus (see \cite{P2, P3} or \cite[page 286]{HR}).

\begin{proposition}\label{transfer}  Suppose $\Psi$ is the conjunction of an ordinal boundable statement and $\BPIT$.  If there is a Fraenkel-Mostowski model of $\Psi$ then there is a Zermelo-Fraenkel model of $\Psi$.
\end{proposition}

\end{subsection}

\end{section}

\end{document}